\documentclass{amsart}

\usepackage{amssymb}
\usepackage{amsfonts}
\usepackage{amsmath}
\usepackage{euscript}
\usepackage{enumerate}
\usepackage{graphics}
\usepackage[active]{srcltx}

\newtheorem{theorem}{Theorem}[section]
\newtheorem{lemma}[theorem]{Lemma}
\newtheorem{proposition}[theorem]{Proposition}
\newtheorem*{Theorem1}{Theorem 1}
\newtheorem*{Theorem2}{Theorem 2}
\newtheorem*{Theorem1'}{Theorem 1'}

\theoremstyle{definition}

\theoremstyle{remark}

\numberwithin{equation}{section}

\newcommand \g{{\mathfrak g}}
\newcommand  \s{{\mathfrak s}}
\renewcommand \r{{\mathfrak r}}
\def \n{{\mathfrak n}}
\renewcommand \t{{\mathfrak t}}
\newcommand \D{{\mathcal C}}
\newcommand \gl{{\mathfrak {gl}}}

\newcommand \B{{\mathcal B}}
\newcommand \GL{{\mathrm {GL}}}

\begin{document}

\title{Jordan-Chevalley decomposition\\ in finite dimesional Lie algebras}

\author{Leandro Cagliero}
\address{CIEM-CONICET, FAMAF-Universidad Nacional de C\'ordoba, C\'ordoba, Argentina.}
\curraddr{Math. Department, MIT, Cambridge, MA 02139-4307, USA.}
\email{cagliero@famaf.unc.edu.ar}
\thanks{The first author was supported in part by CONICET and SECYT-UNC grants.}

\author{Fernando Szechtman}
\address{Department of Mathematics and Statistics, Univeristy of Regina, Canada}
\email{fernando.szechtman@gmail.com}
\thanks{The second author was supported in part by an NSERC discovery grant}





\begin{abstract}
Let $\g$ be a finite dimensional Lie algebra over
a field $k$ of characteristic zero.
An element $x$ of $\g$ is said to have an \emph{abstract Jordan-Chevalley decomposition}
if there exist unique $s,n\in\g$ such that $x=s+n$, $[s,n]=0$ and given any
finite dimensional representation $\pi:\g\to\gl(V)$
the Jordan-Chevalley decomposition of $\pi(x)$ in $\gl(V)$ is $\pi(x)=\pi(s)+\pi(n)$.

In this paper we prove that $x\in\g$ has an abstract Jordan-Chevalley decomposition
if and only if $x\in [\g,\g]$, in which case its semisimple and nilpotent parts are
also in $[\g,\g]$ and are explicitly determined. We derive two immediate consequences:
(1) every element of $\g$ has an abstract Jordan-Chevalley decomposition if and only if
$\g$ is perfect; (2) if $\g$ is a Lie subalgebra of $\gl(n,k)$ then
$[\g,\g]$ contains the semisimple and nilpotent parts of
all its elements. The last result was first proved by Bourbaki using different methods.

Our proof only uses
elementary linear algebra and basic results on the representation theory
of Lie algebras, such as the Invariance Lemma and Lie's Theorem,
in addition to the fundamental theorems of Ado and Levi.
\end{abstract}

\maketitle

\section{Introduction}
Let $k$ be a field of characteristic zero.
All Lie algebras and representations are assumed to be finite dimensional.
An element $x$ of a Lie algebra $\g$ is said to have an \emph{abstract Jordan-Chevalley decomposition}
if there exist unique $s,n\in\g$ such that $x=s+n$, $[s,n]=0$ and
this is compatible with every representation of $\g$, in the sense that
given any finite dimensional representation $\pi:\g\to\gl(V)$
the Jordan-Chevalley decomposition of $\pi(x)$ in $\gl(V)$ is $\pi(x)=\pi(s)+\pi(n)$.
The Lie algebra $\g$ itself is said to have an abstract Jordan-Chevalley decomposition
if everyone of its elements does. As is well-known all semisimple Lie algebras possess
this property. A proof of this fact can be found in any standard book on Lie algebras
(see for instance\cite{Bo1}, \cite{FH} or \cite{Hu}).
As a consequence, if $\pi:\g\to\gl(V)$ is a representation of a
semisimple Lie algebra and $x\in\g$,
the semisimple and nilpotent parts of $\pi(x)$ in $\gl(V)$
actually belong to $\pi(\g)$.

More generally, if $\g\subset\gl(V)$ is a Lie subalgebra of $\gl(V)$ then,
according to Bourbaki (Ch.\ VII, \S5 in \cite{Bo7}),
$\g$ is said to be \emph{decomposable} if for any $x\in\g$
the semisimple and nilpotent parts of $x$ in $\gl(V)$
belong to $\g$.
In particular, every semisimple Lie subalgebra of $\gl(V)$
is decomposable. It is also a classical result that
the Lie algebra of any algebraic subgroup of $\text{GL}(V)$
is decomposable, see for instance Ch.\ I \S4 in \cite{Bo}.
Moreover, Corollary 2 of Ch.\ VII, \S5.5 in \cite{Bo7} states that $[\g,\g]$ is
decomposable for any  Lie subalgebra $\g$ of $\gl(V)$.
In particular, perfect Lie  subalgebras of $\gl(V)$ are decomposable.
It must be pointed out, however, that given two
decomposable Lie subalgebras $\g_i\subset\gl(V_i)$, $i=1,2$, and a
homomorphism $f:\g_1\to\g_2$,
it is not necessarily true that $f$ transforms
semisimple (resp. nilpotent) elements of $\g_1$ to
semisimple (resp. nilpotent) elements of $\g_2$
(a counterexample already appears with $\dim\g_i=1$, $i=1,2$).
This applies in particular to representations of decomposable Lie subalgebras.

This leads naturally to the question of what is the class of Lie algebras
$\g$ that have an abstract Jordan-Chevalley
decomposition.
The answer to this question is given by the following theorem.

\begin{Theorem1}
A Lie algebra $\g$ has an abstract Jordan-Chevalley
decomposition if and only if $\g$ is perfect.
\end{Theorem1}

Although this result is suggested by the statement of
Corollary 2 of Ch.\ VII, \S5.5 in \cite{Bo7}, we could not find an explicit
reference to it in the literature. Nor it is clear that this theorem is
a direct consequence of the results established in
Ch.\ VII, \S5 in \cite{Bo7}.

Theorem 1 as well as Bourbaki's Corollary 2 of Ch.\ VII, \S5.5 in \cite{Bo7}
follow at once from Theorem 2 below. Moreover, Theorem 2 determines
what specific elements of an arbitrary Lie algebra $\g$ admit an abstract Jordan-Chevalley
decomposition, namely those in $[\g,\g]$.
Furthermore, the semisimple and nilpotent parts of any $x\in [\g,\g]$ are shown
 to lie in $[\g,\g]$ as well, and we explicitly indicate how to obtain them.

\begin{Theorem2}
An element $x$ of a Lie algebra $\g$ has an abstract Jordan-Chevalley
decomposition
if and only if $x$ belongs to the derived algebra $[\g,\g]$, in which case
the semisimple and nilpotent parts of $x$ also belong to $[\g,\g]$.
\end{Theorem2}

The proof of Theorem 2 only uses
elementary linear algebra and basic results on representation theory
of Lie algebras, such as the Invariance Lemma and Lie's Theorem
in addition to the fundamental theorems of Ado and Levi.

If $k$ is a perfect field of
positive characteristic, any restricted Lie algebra
admits such an abstract decomposition, known as the
Jordan-Chevalley-Seligman decomposition (see~\cite{Se}).
However, the above theorems are false in general and in fact they are already false for
simple Lie algebras as the following example shows.
Let $k$ be the (perfect) field with 2 elements and let
$\g$ be the 3 dimensional Lie algebra with basis $i,j,k$ satisfying $[i,j]=k$, $[j,k]=i$ and
$[k,i]=j$. Clearly $\g$ is perfect and, being 3 dimensional, it is simple.
However, the semisimple and nilpotent parts of $ad_\g(i)$ in $\gl(\g)$ are easily
seen not to be in $ad(\g)$.

We thank A. Pianzola for pointing out the work of Bourbaki on decomposable Lie algebras,
and J. Vargas for useful discussions on the paper.

\section{Basic preliminaries}
In this section we collect in four lemmas some basic facts that will be used in next section. Here $k$ denotes an arbitrary field.

\medskip

\begin{lemma}\label{Lemma 1.}
 Let $A$ be an upper triangular $n\times n$ matrix over $k$. Then $A$ is diagonalizable
if and only if $A_{ij}=0$ whenever $i<j$ and $A_{ii}=A_{jj}$.
\end{lemma}

\begin{proof}
  $A$ is diagonalizable if and only if the algebraic and geometric
multiplicities of $\alpha$
are the same for every diagonal entry $\alpha$ of $A$.
This means that the leading 1's of the reduced row echelon
form of $A-\alpha I$ all occur in positions $(m,m)$ where
$A_{mm}\neq \alpha$, which is equivalent to $A_{ij}=0$ for all $i<j$ such that $A_{ii}=A_{jj}$.
\end{proof}

\begin{lemma}\label{Lemma 2.}
 Let $A$ be a diagonalizable upper triangular $n\times n$ matrix over $k$.
Then exists $P\in\GL_n(k)$ such that $P$ is upper triangular and $P^{-1}AP$ is diagonal.
\end{lemma}

\begin{proof}
 Replace $A$ by $P^{-1}AP$, where $P=I+\alpha E^{12}$, $\alpha=0$ if $A_{12}=0$,
$\alpha=A_{12}/(A_{22}-A_{11})$ if $A_{12}\neq 0$ (a valid choice by Lemma \ref{Lemma 1.}).
Repeat until $A_{12}=\cdots=A_{1n}=0$ and then reason by induction.
\end{proof}

\begin{lemma}\label{Lemma 3.}
 Let $S:V\to V$ be a diagonalizable endomorphism and let $\B$ be a basis of $V$
such that $A=[S]_\B$ is upper triangular. Let $V_0$ be the 0-eigenspace of $S$
and let $V_*$ be the sum of all other eigenspaces. Let $P:V\to V$ be the projection
with image $V_*$ and kernel $V_0$.
Let $T:V\to V$ be an endomorphism such that $[T]_\B$ is strictly upper triangular.
Then $P(S+T)\vert_{V_*}:V_*\to V_*$
is invertible.
\end{lemma}

\begin{proof}
 By Lemma \ref{Lemma 2.} we may assume that $A$ is already diagonal.
Strike out from $\B$ all vectors from $V_0$.
This produces a basis $\D$ of $V_*$.
The matrix $[PS\vert_{V_*}]_\D$ is diagonal with non-zero
entries and $[PT\vert_{V_*}]_\D$ is strictly upper triangular,
so $P(S+T)\vert_{V_*}$ is invertible.
\end{proof}

The next result is a well-known consequence of the Invariance Lemma
(see Lemma 9.13 of \cite{FH}) and Lie's Theorem (see \cite{FH} or \cite{Hu}),
and is included for the sake of completeness.

\medskip

\begin{lemma}\label{Lemma 4.}
 Suppose $k$ is algebraically closed and of characteristic 0. Let $\g$ be a Lie algebra over $k$ with solvable radical $\r$. Then  $[\g,\r]$ acts trivially on every irreducible $\g$-module.
\end{lemma}

\begin{proof}
  Since $[\g,\r]$ is solvable Lie's theorem implies the existence of a linear functional
$\lambda: [\g,\r]\to k$ such that $V_\lambda=\{v\in V\,\vert\, r\cdot v=\lambda(r)v\text{ for all }r\in [\g,\r]\}$
is non-zero. But $[\g,\r]$ is an ideal of $\g$, so $V_\lambda$ is $\g$-invariant by the Invariance Lemma,
whence $V=V_\lambda$ by the irreducibility of $V$. Let $r\in [\g,\r]$.
Then the trace of $r$ acting on $V$ is 0, while the trace of $r$ acting on $V_\lambda$ is
$\lambda(r)\mathrm{dim}\,V_\lambda$. Hence $\lambda(r)=0$.
(cf. with the proof of Lemma C.19 of \cite{FH}.)
\end{proof}

\section{Determination of the semisimple and nilpotent parts}
The difficult part of Theorem 2 requires the three subsidiary results below, which
explicitly describe how the semisimple and nilpotent parts of a given $x\in[\g,\g]$
can be obtained. Here $k$ stands for an algebraically closed field of characteristic 0.

We consider a Levi decomposition
$\g=\s\ltimes \r$ and let $\n=[\g,\r]$. Note that $[\g,\g]=\s\ltimes \n$.
Moreover, we fix $x\in[\g,\g]$, so $x=a+r$ for unique $a\in\s$ and $r\in \n$,
and let $a=s+n$ be the abstract Jordan decomposition of $a$ in $\s$.
Furthermore, we let $\n_0$ be the 0-eigenspace of $ad_{\g} s$ acting on
$\n$ and denote by $\n_*$ the sum of the remaining eigenspaces.

\medskip

\begin{proposition}\label{Proposition 1.}
 Let $\pi:\g\to\gl(V)$ be a representation.
Then $\pi(n+b)$ is nilpotent for all $b\in \n$.
\end{proposition}

\begin{proof}
 Arguing by induction we are reduced to the case when $V$ is irreducible.
In this case $\pi(b)=0$, by Lemma \ref{Lemma 4.}, and $\pi(n)$ is nilpotent. \end{proof}

\begin{proposition}\label{Proposition 2.}
 Let $\pi:\g\to\gl(V)$ be a representation.
Then $\pi(s+b)$ is diagonalizable for all $b\in \n_*$.
\end{proposition}

\begin{proof}
Consider the solvable subalgebra $\t=k s\oplus \n$ of $[\g,\g]$.
By Lie's theorem and Lemma \ref{Lemma 2.} there is a basis $\D$ of $V$ such that $S=[\pi(s)]_\D$
is diagonal and $[\pi(t)]_\D$ is upper triangular for every $t\in \t$.
Suppose first $b\in\n(\alpha)$, the
$\alpha$-eigenspace of  $ad_{\g} s$ acting on $\n$, $\alpha\neq 0$.
By Proposition \ref{Proposition 1.} $\pi(b)$ is
nilpotent, so $B=[\pi(b)]_\D$ is strictly upper triangular.
We wish to show that $S+B$ is diagonalizable.
By Lemma \ref{Lemma 1.} it suffices to verify that $B_{ij}=0$ whenever $S_{ii}=S_{jj}$.
Since $[s,b]=\alpha b$, we have
$(S_{ii}-S_{jj})B_{ij}=\alpha B_{ij}$, so indeed $B_{ij}=0$ whenever $S_{ii}=S_{jj}$.
In general $b=b_1+\cdots+b_m$,
where $b_i\in\n(\alpha_i)$ and $\alpha_i\neq 0$, so this case follows from the first.
\end{proof}

\begin{proposition}\label{Proposition 3.}
 There exists $b\in\n_*$ such that $[a+r,b]=[s,r]$.
\end{proposition}

\begin{proof}
 Consider the solvable subalgebra $\t=k s \oplus k n\oplus \n$ of $[\g,\g]$.
Argue as in the proof of Proposition \ref{Proposition 2.} to find a basis $\B$ of $\t$ such that $[ad_\t(s)]_\B$
is diagonal and $[ad_\t(n)]_\B, [ad_\t(r)]_\B $ are strictly upper triangular.
Apply Lemma \ref{Lemma 3.} with $V=\t$, $S=ad_\t(s)$ and $T=ad_\t(n)+ad_\t(r)$. Clearly $V_*=\n_*$ and
$V_0=k s\oplus k n\oplus \n_0$. Since $[s,r]\in \n_*$, Lemma \ref{Lemma 3.} ensures the existence of $b\in \n_*$
such that $[a+r,b]=[s,r]+c$, where $c=[a+r,b]-[s,r]\in V_0\cap \n=\n_0$. We next show that $c=0$.
Indeed, by Ado's theorem $\g$ has faithful representation $\pi:\g\to\gl(n,k)$. We have
$[\pi(s)+\pi(b),-\pi(n)-\pi(r)+\pi(b)]=\pi(c)$ (*). By Lie's theorem there is a basis
$\D$ of $\pi(\t)$ such that
$[ad_{\pi(\t)}(\pi(t))]_\D$ is upper triangular for every $t\in\t$.
Apply Lemma \ref{Lemma 3.} with $V=\pi(\t)$, $S=ad_{\pi(\t)}(\pi(s)+\pi(b))$ and $T=ad_{\pi(\t)}(-\pi(b))$.
Here
$S$ is diagonalizable by Proposition \ref{Proposition 2.}. By Lemma \ref{Lemma 3.}
$\mathrm{ker}(S+T)\cap V_*=\{0\}$. But $\pi(c)\in \mathrm{ker}(S+T)$ since $c\in \n_0$, and
$\pi(c)\in V_*$ applying the fact that $S$ is diagonalizable to (*). Hence $\pi(c)=0$, so $c=0$.
\end{proof}

\section{Proof of Theorem 2}

Throughout the proof $k$ is a field of characteristic 0.

\begin{proof}[Necessity.] This is clear since any linear
map from $\g$ to $\gl(V)$ such that $\dim\pi(\g)=1$ and $\pi([\g,\g])=0$
is a representation.

\smallskip

\noindent{\it Sufficiency.} Suppose first that $k$ is
algebraically closed. Let $x\in [\g,\g]$. Fix a Levi decomposition
$\g=\s\ltimes \r$ and let $\n=[\g,\r]$. Then $x=a+r$ for unique
$a\in\s$ and $r\in \n$. Let $a=s+n$ be the abstract
Jordan-Chevalley decomposition of $a$ in $\s$ and set
$\n_*=[s,\n]$. By Proposition \ref{Proposition 3.} there exists
$b\in\n_*$ such that $[a+r,b]=[s,r]$. This translates into
$[s+b,n+r-b]=0$. Clearly $x=(s+b)+(n+r-b)$. Moreover, if
$\pi:\g\to\gl(V)$ is a representation then Propositions
\ref{Proposition 1.} and \ref{Proposition 2.} ensure that
$\pi(s+b)$ is diagonalizable and $\pi(n+r-b)$ is nilpotent, as
required.

Suppose next that $k$ is arbitrary. Let $x\in[\g,\g]$.
Let $k_0$ be an algebraic closure of $k$.
Let $\g_0$ be the extension of $\g$ to $k_0$. By above $x$ has an abstract Jordan
decomposition $x=s+n$ in $[\g_0,\g_0]$. Let $\pi:\g\to\gl(n,k)$ be a representation of $\g$.
Let $\pi_0:\g_0\to \gl(n,k_0)$ be its extension to $\g_0$.
Let $\pi(x)=S+N$ be the Jordan decomposition of $\pi(x)$ in $\gl(n,k)$.
The minimal polynomial, say  $p$, of $S$ is a product of distinct monic
irreducible polynomials over $k$.
Since $k$ has characteristic 0, we see that $p$ has distinct roots in $k_0$,
whence $S$ is diagonalizable
over $k_0$. It follows that $\pi(x)=S+N$ is the Jordan decomposition of $\pi(x)$ in $\gl(n,k_0)$.
On the other hand, by above $S=\pi_0(s)\in \pi_0(\g_0)$ and $N=\pi_0(n)\in\pi_0(\g_0)$.
Hence $S,N\in \pi_0(\g_0)\cap \gl(n,k)$. We claim
that this intersection is $\pi(\g)$. Indeed, let $\B$
be a basis of $k_0$ over $k$ containing 1. Every element $T$ of $\pi_0(\g_0)$
is a sum of matrices of the form $\alpha\pi(y)$,
where $\alpha\in \B$ and $y\in \g$. By the linear independence of $\B$,
if $T$ is also in $\gl(n,k)$ then all summands
are 0 but $1\cdot \pi(y)$, as required. This proves that $\pi_0(s),\pi_0(n)\in \pi(\g)$.
This holds in particular
when $\pi$ is faithful. Using $\B$ once more we see that $\pi_0$
is faithful. It follows that $s,n\in\g\cap[\g_0,\g_0]=[\g,\g]$ and, a fortiori, the Jordan decomposition of
$\pi(x)$ in $\gl(n,k)$, for $\pi$ arbitrary,
is $\pi(x)=\pi(s)+\pi(n)$.

The uniqueness of $s$ and $n$ now follows from the uniqueness of the Jordan
decomposition of $\pi(x)$ in $\gl(n,k)$
when $\pi$ is faithful. This completes the proof of the theorem.
\end{proof}


\end{document}